\documentclass[a4paper]{article}

\usepackage{amsmath}
\usepackage{amssymb}
\usepackage{graphicx}
\usepackage{amsthm}
\usepackage{mathtools}
\usepackage{graphicx}
\usepackage{mathrsfs}
\usepackage{amsfonts}
\usepackage{enumerate}
\usepackage{latexsym, euscript, epic, eepic, color}
\usepackage[pagebackref]{hyperref}
\usepackage{authblk}
\usepackage{blindtext}

\newtheorem{lemma}{Lemma}[section]
\newtheorem{theorem}{Theorem}[section]
\newtheorem{definition}{Definition}[section]
\newtheorem{proposition}{Proposition}[section]

\newtheorem{corollary}{Corollary}[section]

\newtheorem{remark}{Remark}[section]

\numberwithin{equation}{section}

\def\upp#1{{\overline{#1}}}
\def\low#1{{\underline{#1}}}

\def\floor#1{\left\lfloor #1 \right\rfloor}
\def\ceil#1{\left\lceil #1 \right\rceil}
\def\brackets#1{\left( #1 \right)}
\def\medbrackets#1{\left[ #1 \right]}
\def\bigbrackets#1{\left\{ #1 \right\}}
\def\length#1{\left| #1 \right|}
\def\absolutevalues#1{\left| #1 \right|}

\def\remark#1{\noindent {\bf Remark.} #1}

\newcommand{\NN}{\mathbb{N}}
\newcommand{\RR}{\mathbb{R}}

\newcommand{\eps}{\varepsilon}
\newcommand{\mathdot}{~:~}
\newcommand{\upplim}{\varlimsup}
\newcommand{\lowlim}{\varliminf}
\newcommand{\proj}{\text{Proj}}
\newcommand{\CE}{Chung-Erd\H{o}s }
\newcommand{\DS}{Duffin-Schaeffer }

\title{Dimensions of the popcorn graph}

\author{Haipeng Chen$^1$, Jonathan M. Fraser$^2$, and Han Yu$^3$}
\affil{$^1$College of Big Data and Internet, Shenzhen Technology University, \\ Shenzhen, China \\
e-mail: hpchen0703@foxmail.com}
\affil{$^2$School of Mathematics and Statistics, University of St Andrews, UK \\
	e-mail: jmf32@st-andrews.ac.uk}
\affil{$^3$Department of Pure Mathematics and Mathematical Statistics, \\ University of Cambridge, UK \\
	e-mail: hy351@math.cam.ac.uk}

\begin{document}
	
	\maketitle
	
	\pagenumbering{arabic}

	\begin{abstract}
		The `popcorn function' is a well-known and important example in real analysis with   many interesting features.  We prove that the box dimension of the graph of the popcorn function is 4/3, as well as computing the Assouad dimension and Assouad spectrum.  The main ingredients include Duffin-Schaeffer type estimates from  Diophantine approximation and the Chung-Erd\H{o}s inequality from probability theory.
		\\ \\
		\emph{Key words and phrases:} popcorn function, Thomae function, box dimension, Assouad dimension, Assouad spectrum.\\
		\emph{Mathematics Subject Classification 2010:} primary: 28A80; secondary: 11B57.
	\end{abstract}

	\section{The popcorn function}
	
	The \emph{popcorn function} $f: [0,1] \to \mathbb{R}$ is defined by
	\begin{equation}
	f(x) = 
	\begin{cases}
	\frac{1}{q} & \text{ if $x = \frac{p}{q}$ where $\gcd{(p,q)}=1$ and $1 \leq p < q$} \\
	0 & \text{ otherwise.}
	\end{cases}
	\end{equation}
	 It is an important pedagogical example in real analysis and is  also known as {\it Thomae's function, the   raindrop function, and the modified Dirichlet function}, etc. This function has many interesting properties, such as being continuous at the  irrationals but discontinuous at the  rationals in $(0,1)$.  It also provides an example of a Riemann integrable function which is not continuous on any open interval.  We write 
	\[
	G_f := \{(x,f(x)) ~:~ x \in [0,1] \}
	\]
	to denote the  \emph{graph} of the popcorn function, which we refer to as the {\it popcorn graph}. The restriction of the popcorn graph to the rationals provides a simple example of a discrete set whose closure has positive length.
	
	We also define the {\it full popcorn set} by
	\begin{equation}
	F = \left\{(p/q, 1/q) : p, q \in \mathbb{N} \text{ with } p <q\right\} \cup ([0,1] \times \{0\})
	\end{equation} 
	It is clear that $G_f \subset  F  \subset [0,1]^2.$  We will see below that the popcorn graph and full popcorn set have the same dimensions and so we include the full popcorn set in our analysis for completeness.
	
	\begin{figure}[htbp]
		\centering
		\begin{minipage}[t]{0.49\textwidth}
			\centering
			\includegraphics[width=0.9\textwidth]{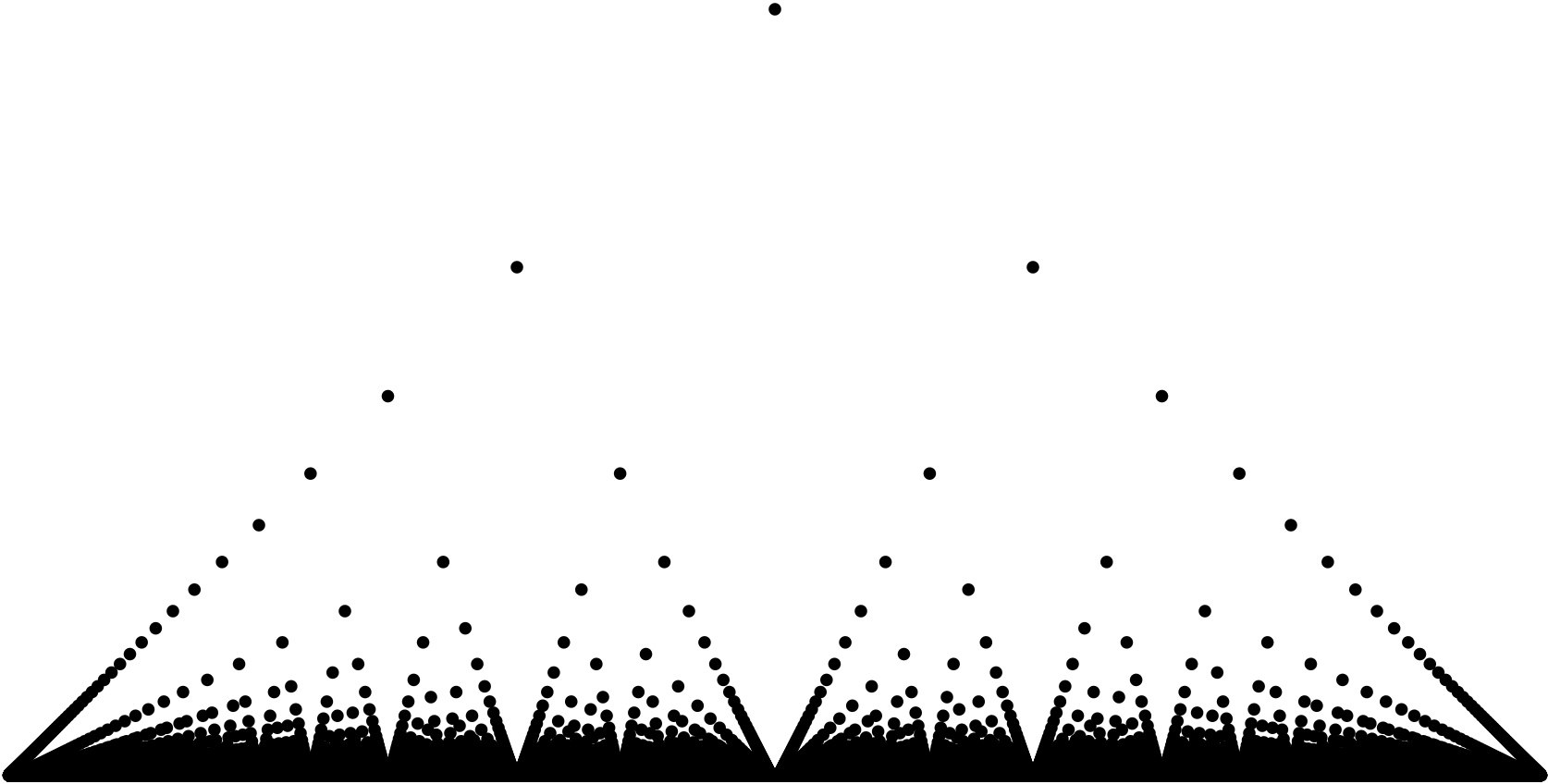}
			\caption{Popcorn graph}
		\end{minipage}
		\begin{minipage}[t]{0.5\textwidth}
			\centering
			\includegraphics[width=0.9\textwidth]{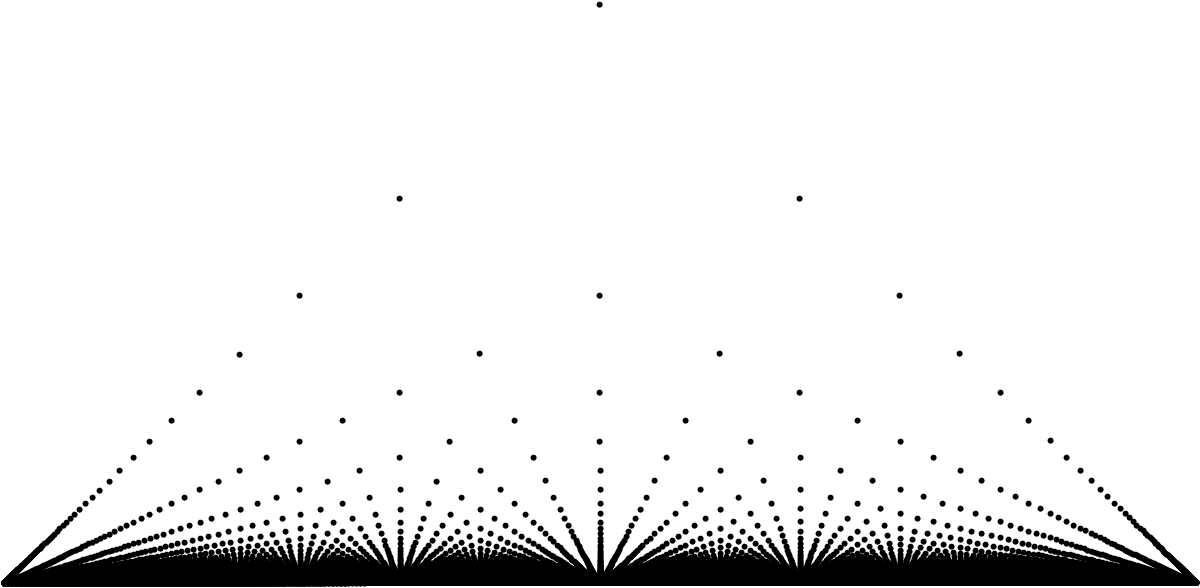}
			\caption{Full popcorn set}
		\end{minipage}
	\end{figure}
	\section{Box dimension of the popcorn graph}
	
	In this section, we discuss the box dimension of the popcorn graph. We start with the definition of box dimension. More discussion on the definition and properties of box dimension can be found in \cite{F2004}.  We remark that the \emph{Hausdorff} dimensions of the popcorn graph and full popcorn set are trivially 1 since Hausdorff dimension is is countably stable. 
	
	\begin{definition}
		Let $X \subset \RR^d$ be a non-empty bounded set.  For any $0 < \delta < \length{X}$, where $\length{X}$ is the diameter of $X$, we write $N_\delta(X)$ to denote  the smallest number of closed cubes with side length $\delta$ needed to cover the set $X$. The \emph{upper and  lower box dimensions} of $X$ are defined, respectively, as 
		\[
		\upp{\dim}_B X  = \upplim_{\delta \to 0} \frac{\log N_{\delta} (X)}{-\log \delta}; \qquad
		\low{\dim}_B X = \lowlim_{\delta \to 0} \frac{\log N_{\delta} (X)}{-\log \delta} .
		\]
If the upper and lower box dimensions coincide, then we write $\dim_B X$ for the common value, referring simply to the \emph{box dimension} of $X$.
	\end{definition}
  We will prove the following result in Section \ref{proof1}.
\begin{theorem}\label{dim_B_result}
	The box dimensions of the popcorn graph and full popcorn set are $4/3$, that is, $\dim_B F  = \dim_B G_f = 4/3.$ \end{theorem}
Both the popcorn graph and full popcorn set have a fractal structure.  In fact, they can be used to exhibit some interesting phenomena in fractal geometry, for example, that the modified lower dimension is not stable under closure, see \cite[Section 3.4.2]{F2020}. Despite its sustained relevance and appearance in analysis and fractal geometry,  the box dimension of the popcorn graph was  unknown.   The proof we found relies on a delicate  counting argument introduced   by Duffin and Schaeffer in \cite{DS1941}. We also make extensive use of the  Chung-Erd\H{o}s inequality from probability theory.    The popcorn graph is clearly related to the set $E=\{1/n : n \in \mathbb{N}\}$ which is often used as a `first example' when studying the box dimension of fractals.  It is straightforward, but instructive, to show that  $	{\dim}_B E = 1/2.$ Computation of the box dimension of the popcorn graph is, by comparison, rather harder.

	\section{Assouad spectrum and Assouad dimension of the  popcorn graph} \label{assouadsection}
In this section, we discuss the Assouad spectrum  and Assouad dimension of the popcorn graph. The Assouad dimension can be viewed as a `local box dimension' where only covers of small neighbourhoods of the set are considered.  The Assouad spectrum fixes the relationship between the size of the neighbourhood and the covering scale using the  parameter $\theta \in (0,1)$ and as the parameter varies the Assouad spectrum interpolates between the box and Assouad dimensions in a meaningful sense explained below.  It is therefore very natural to also consider the Assouad spectrum of the popcorn graph. 
\begin{definition}
	Let $X \subset \RR^d$ be a non-empty  set. For any $x = \brackets{x_1, \dots, x_d} \in \RR^d$ and any $R > 0$, we write $C(x,R) = \medbrackets{x_1,x_1+R} \times\dots \times \medbrackets{x_d,x_d+R}$.	The Assouad dimension of $X$ is defined by
	\[
	\begin{aligned}
	\dim_A X = \inf \{ s \geq 0  \mathdot & \text{there exists a constant $c > 0$, such that} \\
	& \hspace{-1cm}\text{ for all }0 < r < R \text{ and all } x \in \RR^d, 
	N_r(C(x,R) \cap X) \leq c \brackets{R/r}^s \}.
	\end{aligned}
	\]
	For  $0 < \theta < 1$, the Assouad spectrum of $X$ is defined by
	\[
	\begin{aligned}
	\dim_A^\theta X = \inf \{ s \geq 0  \mathdot & \text{there exists a constant $c > 0$, such that } \\
	& \hspace{-2cm} \text{ for all } 0 < R < 1 \text{ and all } x \in \RR^d, N_{R^{\frac{1}{\theta}}}(C(x,R) \cap X) \leq c \,  R^{(1-\frac{1}{\theta})s} \}.
	\end{aligned}
	\]
\end{definition}
	We refer the reader to \cite{F2020,FHHTY2018,FY2016A} for more details and background on the  Assouad dimension and spectrum.  We note that $\dim_A^\theta X $ is continuous in $\theta \in (0,1)$ and tends to the upper box dimension as $\theta \to 0$ and to the quasi-Assouad dimension as $\theta \to 1$.  The quasi-Assouad dimension,  introduced in \cite{luxi} and denoted by $\dim_{qA}$, is related to the Assouad dimension, and for many sets of interest, the two notions coincide.   This will be the case here.  It is also useful to note that for all non-empty bounded $X$ and all $\theta \in (0,1)$
	\begin{equation} \label{dimss}
	\upp{\dim}_B X \leq \dim_A^\theta X \leq \dim_{qA} X  \leq \dim_A X.
	\end{equation}
	 The techniques we used to deal with  the box dimension may also be used to study the Assouad spectrum, but the argument becomes rather more complicated. 	We will  prove the following result in Section \ref{proof2}.
	\begin{theorem}\label{A_spe_result}
		The Assouad spectrum of the popcorn graph is
		\[
		\dim_A^\theta G_f = \dim_A^\theta  F  =
		\begin{cases}
		\frac{\frac{4}{3}-\theta}{1-\theta} & \theta \in \left(0, \frac{2}{3} \right)\\
		2 & \theta \in [\frac{2}{3}, 1)
		\end{cases}
		\]
	\end{theorem}
	We get the following immediate corollary concerning Assouad and quasi-Assouad dimension by appealing to \eqref{dimss}.
	\begin{corollary}
		We have $\dim_{\text{q}A} G_f = \dim_{\text{q}A}  F  = \dim_A G_f = \dim_A  F  = 2.$
	\end{corollary}
	\section{Preliminaries and notation}

	\subsubsection*{Computing the box dimension}
	Suppose  $\bigbrackets{\delta_n}_{n=1}^{\infty}$ is a strictly positive decreasing sequence of  real numbers  such that there exists a constant $0 < c < 1$ such that for any $n \geq 1$, we have $\frac{\delta_{n+1}}{\delta_n} \geq c > 0$. It is straightforward to show that the box dimensions for $X\subset\mathbb{R}^d$ can be computed by 
	\[
	\upp{\dim}_B X  = \upplim_{n \to \infty} \frac{\log N_{\delta_n} (X)}{-\log \delta_n}; \qquad
	\low{\dim}_B X  = \lowlim_{n \to \infty} \frac{\log N_{\delta_n} (X)}{-\log \delta_n},
	\]
	that is, it is sufficient to let $\delta \to 0$ through the sequence $\delta_n$. 
	\subsubsection*{Layer structures of the popcorn graph}
		It will be useful to keep in mind the following two expressions for the popcorn graph. In what follows, observe that $\gcd(1,1)=1.$  First note that we need only to consider the popcorn graph restricted to the rationals since the box dimension (and Assouad spectrum) is stable under taking closure.  The `horizontal view' of the popcorn graph (restricted to the rationals) is
	\[
	G_f \cap \brackets{[0,1] \times (0,1]} = \bigcup_{n=2}^{\infty} \bigbrackets{ \brackets{\frac{i}{n}, \frac{1}{n}} \mathdot \gcd{(i,n) = 1}, 1 \leq i \leq n}
	\]
	and the `collapsed view' is
	\[
	G_f \cap \brackets{\medbrackets{0,1/2}\times (0,1]}  = \bigcup_{l=1}^{\infty}  \bigcup_{n=1}^{\infty} \bigbrackets{ \brackets{\frac{l}{ln+i}, \frac{1}{ln+i}} \mathdot \gcd{(i,l) = 1}, 1 \leq i \leq l}.
	\]
	This expression uses that if $\gcd{(i,l)} = 1$, then   $\gcd{(l,ln + i)} = 1$ for every integer $n \geq 1$. 	These two expressions are useful in different settings.  The horizontal view will be used to study the box dimensions of the popcorn graph, and the collapsed view will be used to study the Assouad spectrum of the popcorn graph. 
	
	\subsubsection*{Further notation}
	\begin{itemize}
		\item 	Throughout, we write  $a \lesssim b ~ ( \text{or } a \gtrsim b )$ to express $a \leq  cb~ (\text{or } a \geq cb)$ for some universal constant $c$. If $a \lesssim b$ and $a \gtrsim b$, then we write  $a \approx b$. 
		\item 	For  $x > 1$, we write  $\floor{x} = \max \{ n \in \mathbb{N}^+ ~:~ n \leq x \}$ and $\ceil{x} = \min \{ n \in \mathbb{N}^+ ~:~ n \geq x \}$. Observe that if $a > b > 1$ are  two real numbers with $a-b \geq 3$, then 
		\begin{equation}\label{square_esti}
		\floor{a}^2 - \ceil{b}^2 \approx a^2 - b^2. 
		\end{equation}
		\item 
		For any set $X \subset \RR^2$, we write $\proj_x(X)$ to denote the projection of $X$ onto the $x$-axis.
		
	\end{itemize}

	\section{Box dimensions: proof of Theorem \ref{dim_B_result}} \label{proof1}
	
	\subsection{Preparation}

	In this section, we introduce some notation which is specific to the box dimension argument, as well as recall some crucial estimates which we will rely on.   
	
	Let $0 < \delta < 1$ and,  for any integer  $0 \leq k \leq \floor{\delta^{-1/2}}$,  write
	\begin{align}
	S_\delta(k,f) & = \brackets{[0,1] \times [k\delta, (k+1)\delta)} \cap G_f \label{dim_B_S_delta}\\
	S_\delta(k,F) & = \brackets{[0,1] \times [k\delta, (k+1)\delta)} \cap  F.   \label{dim_B_S_delta_full}
	\end{align}
	It is worth noting that $S_\delta(k, f)$ is the $k$th  strip of height $\delta$ in the popcorn graph (and similar for  full popcorn set). We also let
	\begin{equation}\label{L_delta}
	L_\delta(k) = \max \bigbrackets{n \mathdot n \leq \frac{1}{k\delta}} = \floor{\frac{1}{k\delta}}.
	\end{equation}
The  numbers $L_\delta(k) $ effectively index the strips $S_\delta(k, f)$. For  $0 < \delta < 1$, it is clear to see that $L_{\delta}(k) \geq L_{\delta}(k+1)$ for any $k \geq 1$. Moreover,  for any $k \geq 0$, 
	\[
	\ceil{\frac{1}{(k+1)\delta}} = 
	\begin{cases}
	L_\delta(k+1) + 1 & \text{if } \frac{1}{(k+1)\delta} \text{ is not an integer} \\
	L_\delta(k+1) & \text{otherwise.}
	\end{cases}
	\] 
	Thus for each strip $S_\delta(k,f)$ (or $S_\delta(k,F)$), the levels $n \in \NN$ satisfy
	\[
	k\delta \leq \frac{1}{n} < (k+1)\delta\quad \Longleftrightarrow \quad
	L_\delta(k+1) < n \leq L_\delta(k).
	\]
	Let $\psi \mathdot \NN \to \RR$ be a real function. Let $k \geq 1$,   $\delta > 0$ and  for  $L_\delta(k+1) < n \leq L_\delta(k)$ we write 
	\begin{equation}\label{E_n}
	E_n = \bigbrackets{ x \in [0,1] ~:~ \absolutevalues{x - \frac{m}{n}} \leq \frac{\psi(n)}{n} \text{ for some }(m,n), \gcd{(m,n)} = 1, 1 \leq m \leq n-1}.
	\end{equation}
	The set $E_n$   is a finite union of intervals of  length $\frac{2\psi(n)}{n}$.  For any integer $n \geq 2$, we denote the {\it Euler totient function} by
	\[
	\phi(n) = \# \bigbrackets{m ~:~ \gcd{(m,n)} = 1, 1 \leq m \leq n-1}.
	\]
	We require the following bound on the growth of the Euler totient function.
	\begin{theorem}[Estimate of Euler totient function] \cite[Theorem 2.9]{MV2006}\label{totient_funcion}
		
		There exists a constant $0 < c < 1$ and an  integer $N$ such that, for all $n > N$, 
		\[
		\phi(n) \geq c \cdot \frac{n}{\log \log n}.
		\]
	\end{theorem}
	
	We will frequently use the following crude  estimate on  $\log \log n$ for large $n$.  Specifically, for all  $\eps > 0$, there exists an integer $N(\eps) > 0$ such that, for all $n > N$, 
	$
	\log \log n < n^\eps.
	$
	
	We use the Chung-Erd\H{o}s Inequality from probability theory to provide the lower bound for the size of covers by intervals. 
	
	\begin{theorem}[Chung-Erd\H{o}s Inequality]\cite{CE1952, P1995}\label{C-E_ine}
		Let $\bigbrackets{X,\mu,\mathcal{X}}$ be a probability space, $A_1, \dots , A_m$ be positive events in $\bigbrackets{X,\mu,\mathcal{X}}$, then
		\[
		\mu(A_1 \cup \dots \cup A_m ) \geq \frac{\brackets{\sum\limits_{i=1}^m \mu(A_i)}^2}{\sum\limits_{i=1}^m \sum\limits_{j=1}^m \mu(A_i \cap A_j)}.
		\]
	\end{theorem}
	
	In what follows, we let $X = [0,1]$, $\mu$ be the Lebesgue measure on $[0,1]$ and $\mathcal{X}$ be the class of Lebesgue measurable sets in $[0,1]$. 
	
	\begin{theorem}[Duffin-Schaeffer estimate]\cite[Lemma 2]{DS1941}\label{D-S_esti}
		Let $m,n$ be two positive integers satisfying $m, n \geq 2$ and $m \neq n$.  Then  
		\[
		\mu(E_n\cap E_m) \leq 4 \psi(n) \psi(m)
		\]
		where $E_n$ is as  in (\ref{E_n}).
	\end{theorem}
	
	For the rest of the paper, for fixed $\delta$, we let $\psi(n) = n \delta$ in  the definition of $E_n$ (see (\ref{E_n})) for all  integers $n \geq 2$.
	
	\begin{lemma}[Counting integers in horizontal strips]\label{integer_in_strip_k} Fix sufficiently small $0 < \delta < 1$ and sufficiently small $\eps > 0.$ Then for all $1 \leq k \leq \floor{\delta^{-1/2 + \eps}}$, we have
		\[
		\frac{1}{2k^2\delta} \leq L_{\delta}(k) - L_{\delta}(k+1) \leq \frac{1}{k^2\delta}
		\]
		where $L_{\delta}(k)$ is as in (\ref{L_delta}).
	\end{lemma}
	\begin{proof}
		For the upper  bound,
		\[
		L_{\delta}(k) - L_{\delta}(k+1)=  \floor{\frac{1}{k\delta}} - \floor{\frac{1}{(k+1)\delta}} \leq \frac{1}{k\delta} - \frac{1}{(k+1)\delta} + 1 \leq \frac{1}{k^2\delta}
		\]
		for sufficiently small $\delta>0$.  For the lower bound,
		\[
		L_{\delta}(k) - L_{\delta}(k+1) =  \floor{\frac{1}{k\delta}} - \floor{\frac{1}{(k+1)\delta}} \geq \frac{1}{k\delta} - \frac{1}{(k+1)\delta} - 1 \geq \frac{1}{(k+1)^2\delta} - 1 \geq \frac{1}{2k^2\delta},
		\]
		for sufficiently small $\delta>0$, as required. 
	\end{proof}
	
	\remark{The inclusion  of $\eps$ in the range of allowable $k$ in Lemma \ref{integer_in_strip_k} is to guarantee that $\frac{1}{(k+1)^2\delta} - 2 > 1$. We also have $L_{\delta}(k) - L_{\delta}(k+1) \approx \frac{1}{k^2\delta}$.}

	\begin{corollary}\label{cor_integer_in_strip_k}
		Fix sufficiently small $0 < \delta < 1$ and sufficiently small $\eps > 0$. For $1 \leq k \leq \floor{\delta^{-1/2 + \eps}}$, we have
		\[
		\sum\limits_{i = L_{\delta}(k+1) + 1}^{L_\delta(k)}  i \approx \frac{1}{k^3\delta^2}.
		\]
	\end{corollary}
	
	The following  lemma provides estimates of the number of cubes required to cover 
	$S_\delta(k,f)$ and $S_\delta(k,F)$. Recall the definitions  (\ref{dim_B_S_delta}) and (\ref{dim_B_S_delta_full}).
	
	\begin{lemma}\label{dim_B_cover_esti}
		Fix sufficiently small $0 < \delta < 1$, sufficiently small $\eps > 0$, and an integer $1 \leq k \leq \floor{\delta^{-1/2 + \eps}}$. For all  $L_\delta(k+1) < m \leq L_\delta(k)$, we have
		\[
		\frac{1}{4\delta} \cdot \mu\brackets{\bigcup\limits_{m = L_{\delta}(k+1) + 1}^{L_{\delta}(k)} E_m} \leq N_\delta \brackets{S_\delta(k, f)} \leq N_\delta  \brackets{S_\delta(k,F)} \leq \sum\limits_{m = L_{\delta}(k+1) + 1}^{L_{\delta}(k)} m.
		\]
	\end{lemma}
	
	\begin{proof}
		The upper bound simply follows by estimating the cardinality of $S_{\delta}(k,F)$. For the lower bound, observe that  $\bigcup_{m = L_{\delta}(k+1) + 1}^{L_{\delta}(k)} E_m$ is a finite union of disjoint intervals, say $I_1, \dots, I_M$. For each interval $I_j$,  the distance between consecutive points in 
		$I_j \cap \proj_x(S_\delta(k,f))$ is no more than $2\delta$. Thus for every interval $I_j$, we need at least $\frac{\mu{(I_j)}}{4\delta}$ cubes to cover. Thus the lower bound holds.
	\end{proof}

	\subsection{Proof of Theorem \ref{dim_B_result}}

	\begin{proof}[Proof of Upper Bound]
		It suffices to bound the upper box dimension of the full popcorn set $F$ from above.  Fix $\delta \in (0,1)$.  We decompose $F$ into two regions which we treat separately. It follows by a simple cardinality estimate  that
		\[
		N_{\delta} \left( F  \cap \brackets{[0,1] \times [\delta^{2/3}, 1]}\right)  \leq \# F  \cap \brackets{[0,1] \times [\delta^{2/3}, 1]} \lesssim \left(\delta^{-2/3}\right)^2= \delta^{-4/3}.
		\]
	For the remaining part of $F$,
		\[
		N_{\delta} \left( F  \cap \brackets{[0,1] \times [0, \delta^{2/3})}\right) \leq N_{\delta}  \brackets{[0,1] \times [0, \delta^{2/3})}  \lesssim \delta^{-1} \cdot \delta^{-1/3} = \delta^{-4/3},
		\]
		as required.
	\end{proof}
	
	\begin{proof}[Proof of Lower Bound]
		Fix $\eps> 0$.  We write $\delta_n = \brackets{\frac{1}{n(n+1)}}^6$ for  $n \geq 1$ throughout this part of the proof. It suffices to prove that  for large enough  $n$,
		\[
		N_{\delta_n} \left(G_f \cap \brackets{[0,1] \times [\delta_n^{2/3}, \delta_n^{1/2}]}\right) \gtrsim \delta_n^{-(4/3) + \eps}.
		\]
There is nothing particularly special about this choice of $\delta_n$, but it is convenient for the  reciprocals of the square and cube roots of $\delta_n$ to be integers, for example.  To prove this, we consider  $N_{\delta_n}(S_{\delta_n}(k,f))$ for $\delta_n^{-1/3} \leq k \leq \delta_n^{-1/2 +\eps}$. Fix integers $n$ and $k$ in this range. It follows from   Lemma \ref{dim_B_cover_esti} and the  \CE inequality (Theorem \ref{C-E_ine}) that
		\begin{equation}\label{dim_B_lower_bound_1}
		N_{\delta_n}(S_{\delta_n}(k,f)) 
		\gtrsim \delta_n^{-1} \cdot \mu\brackets{\sum\limits_{i = L_{\delta_n}(k+1) + 1}^{L_{\delta_n}(k)} E_i}
		\gtrsim \delta_n^{-1} \cdot \frac{\brackets{\sum\limits_{i = L_{\delta_n}(k+1) + 1}^{L_{\delta_n}(k)} \mu(E_i)}^2}{ \mathop{\sum\sum}\limits_{i,j = L_{\delta_n}(k+1) + 1}^{L_{\delta_n}(k)} \mu(E_i \cap E_j)}
		\end{equation}
		where $E_i$ is as in (\ref{E_n}) with $\psi(i) = i \cdot \delta_n$ for all $i \geq 2$.  First, it follows from Theorem \ref{totient_funcion} that 
		\begin{align*}
		\sum\limits_{i = L_{\delta_n}(k+1) + 1}^{L_{\delta_n}(k)} \mu(E_i) &\gtrsim \delta_n \cdot \brackets{\sum\limits_{i = L_{\delta_n}(k+1) + 1}^{L_{\delta_n}(k)} i \cdot \brackets{\log \log i}^{-1}} \\
&\gtrsim \delta_n \cdot \brackets{\log \log \frac{1}{k\delta_n}}^{-1} \cdot \brackets{\sum\limits_{i = L_{\delta_n}(k+1) + 1}^{L_{\delta_n}(k)} i}.
		\end{align*}
		Thus, applying Corollary \ref{cor_integer_in_strip_k},
		\begin{equation}\label{dim_B_lower_bound_2}
		\begin{aligned}
		\brackets{\sum\limits_{i = L_{\delta_n}(k+1) + 1}^{L_{\delta_n}(k)} \mu(E_i)}^2 & \geq \brackets{\sum\limits_{i = L_{\delta_n}(k+1) + 1}^{L_{\delta_n}(k)} i}^2  \cdot \brackets{\log \log \frac{1}{k\delta_n}}^{-2} \cdot \delta_n^2 \\
		& \gtrsim \delta_n^{2+2\eps} \cdot \brackets{\frac{1}{k^3\delta_n^2}}^2.
		\end{aligned}
		\end{equation}
		Second, it follows from Theorem \ref{D-S_esti} that
		\[
		\sum\limits_{i,j = L_{\delta_n}(k+1) + 1, i \neq j}^{L_{\delta_n}(k)}  \mu(E_i \cap E_j) \lesssim \delta_n^2  \sum\limits_{i,j = L_{\delta_n}(k+1) + 1, i \neq j}^{L_{\delta_n}(k)} i \cdot j \lesssim \delta_n^2 \brackets{\sum\limits_{i = L_{\delta_n}(k+1) + 1}^{L_{\delta_n}(k)}  i}^2 .
		\]
		Thus, again applying Corollary \ref{cor_integer_in_strip_k},
		\begin{equation}\label{dim_B_lower_bound_3}
		\begin{aligned}
		 \mathop{\sum\sum}\limits_{i,j = L_{\delta_n}(k+1) + 1}^{L_{\delta_n}(k)}  \mu(E_i \cap E_j)
		& \lesssim \sum\limits_{i,j = L_{\delta_n}(k+1) + 1, i \neq j}^{L_{\delta_n}(k)}  \mu(E_i \cap E_j) + \sum\limits_{i = L_{\delta_n}(k+1) + 1}^{L_{\delta_n}(k)} \mu(E_i) \\
		& \lesssim \delta_n^2 \brackets{\sum\limits_{i = L_{\delta_n}(k+1) + 1}^{L_{\delta_n}(k)} i}^2 + \delta_n \brackets{\sum\limits_{i = L_{\delta_n}(k+1) + 1}^{L_{\delta_n}(k)} i} \\
& \lesssim \delta_n^2 \brackets{\frac{1}{k^3\delta_n^2}}^2 + \delta_n \brackets{\frac{1}{k^3\delta_n^2}} .
		\end{aligned}
		\end{equation}
		Combining (\ref{dim_B_lower_bound_1})-(\ref{dim_B_lower_bound_3}) we get
		\begin{align*}
		N_{\delta_n}(S_{\delta_n}(k,f))   \gtrsim \frac{\delta_n^{1 + 2\eps} \cdot \brackets{\frac{1}{k^3\delta_n^2}}^2 }{\delta_n^2 \brackets{\frac{1}{k^3\delta_n^2}}^2 + \delta_n \brackets{\frac{1}{k^3\delta_n^2}}}  = \delta_n^{-1+ 2\eps} \cdot \frac{1}{k^3\delta_n + 1} 
		\gtrsim \delta_n^{-1+ 2\eps} \cdot \frac{1}{k^3\delta_n}.
		\end{align*}
		Thus summing over  $\delta_n^{-1/3} \leq k \leq \delta_n^{-1/2 +\eps}$, it follows from (\ref{square_esti}) that
		\[
		N_{\delta_n}(G_f \cap \brackets{[0,1] \times [\delta_n^{2/3}, \delta_n^{1/2}]}) \gtrsim \delta_n^{-1+ 2\eps} \cdot  \sum\limits_{k=\delta_n^{-1/3}}^{\delta_n^{-1/2 + \eps}} \frac{1}{k^3\delta_n} \gtrsim \delta_n^{- \frac{4}{3}+2\eps}.
		\]
		This proves a lower bound of $4/3-2\eps$ and the result follows by letting $\eps$ tend  to 0.
	\end{proof}
	
	\section{Assouad spectrum: proof of Theorem \ref{A_spe_result}} \label{proof2}

	\subsection{Preparation}

	In this section we introduce some notation which is specific to the Assouad spectrum argument, as well as recall some crucial estimates which we will rely on.   
	
	For integers $n, l \geq 1$, and  real numbers $0 < \delta < 1$, we introduce the following notation.  We write
	\begin{equation}\label{A_spe_S_l,n}
	S(l,n)  = \bigbrackets{ \brackets{\frac{l}{ln+i}, \frac{1}{ln+i}} \mathdot \gcd(i,l) = 1, \, 1 \leq i \leq l-1},
	\end{equation}
	that is,  the points in the popcorn graph which lie  on the line $y = \frac{x}{l}$ with $\frac{1}{n+1} < x < \frac{1}{n}$.  We also write 
	\begin{equation}\label{A_spe_proj_S_l,n}
	\proj_x (S(l,n)) = \bigbrackets{ \frac{l}{ln+i} ~:~ \gcd(i,l) = 1, \, 1 \leq i \leq l-1}
	\end{equation}
	for the projection of $S(l,n)$ onto the $x$-axis, and 
	\begin{equation}\label{A_spe_F_S_l,n}
	F_{S(l,n)}(\delta)  = \bigbrackets{ x \in [0,1] \mathdot \absolutevalues{x- \frac{l}{ln + i}} \leq \delta \text{ for some }i, \gcd{(l,i)}= 1, 1 \leq i \leq l-1} 
	\end{equation}
	to denote the natural cover of $\proj_x \brackets{S(l,n)}$ by  cubes of side length $2\delta$.  Similar to  $L_{\delta}(k)$ (see (\ref{L_delta})),  we write 
	\begin{equation}\label{A_spe_L'_delta,n}
	L'_{\delta, n}(k)  = \max \bigbrackets{ m \mathdot m \leq \frac{1}{k(n+1)\delta}} = \floor{ \frac{1}{k(n+1)\delta} }
	\end{equation}
	to index the lines separating the collapsed strips of level $k$. Also, similar to  $L_{\delta}(k)$,  
	\[
	\ceil{\frac{1}{(k+1)(n+1)\delta}} = 
	\begin{cases}
	L'_{\delta, n}(k + 1) + 1 & \text{if } \frac{1}{(k+1)(n+1)\delta} \text{ is not an integer}\\
	L'_{\delta, n}(k + 1)  & \text{otherwise.}
	\end{cases}
	\]
	For simplicity, for fixed $n$ we write  
	\begin{equation}\label{A_spe_strip_1}
	S_l : = S(l,n); \qquad F_l(\delta) : = F_{S(l,n)}(\delta); \qquad L'_{\delta}(k) : = L'_{\delta, n}(k).
	\end{equation}
	For  $0 < \theta < 1$, we write
	\begin{equation}\label{delta_n_theta}
	\delta_n(\theta) = \brackets{\frac{1}{n(n+1)}}^{\frac{1}{\theta}},
	\end{equation}
	and when studying the covers with side length $\delta_n(\theta)$ for some $0 < \theta < 1$ and sufficiently large integer $n$, we write
	\begin{equation}\label{A_spe_strip_2}
	\begin{aligned}
	S_{\delta_n(\theta),\theta}(k) = \bigcup_{l = L'_{\delta_n(\theta)}(k+1) + 1}^{L'_{\delta_n(\theta)}(k)} S_l, & \quad 
	F_{\delta_n(\theta),\theta}(k) = \bigcup_{l = L'_{\delta_n(\theta)}(k+1) + 1}^{L'_{\delta_n(\theta)}(k)} F_l(\delta_n(\theta)), \\
	\proj_x \brackets{S_{\delta_n(\theta),\theta}(k)} & = \bigcup_{l = L'_{\delta_n(\theta)}(k+1) + 1}^{L'_{\delta_n(\theta)}(k)} \proj_x \brackets{S_l}.
	\end{aligned}
	\end{equation}
	
	\begin{proposition}[Local Duffin-Schaeffer Estimate]\label{local_D-S_esti}
		For sufficiently large  $n$, for all $l ,l' \geq 2$ with $l \neq l'$, we have
		\[
		\mu{(F_l(\delta) \cap F_{l'}(\delta))} \leq 8 l l' \delta^2 (n+1)^2.
		\]
	\end{proposition}
	
	\begin{proof}
		Fix $l \neq l'$ and $1 \leq i \leq l-1$ and $1 \leq i' \leq l'-1$.  Notice that
	\[
	\frac{1}{(n+1)^2}\absolutevalues{\frac{i}{l} -\frac{i'}{l'}} =	\frac{\absolutevalues{li'-l'i}}{ll'(n+1)^2} \leq \absolutevalues{\frac{li'-l'i}{(ln+i)(l'n+i')}} = \absolutevalues{\frac{l}{ln+i} - \frac{l'}{l'n+i'}}.
		\]
Thus, to bound $\mu{(F_l(\delta) \cap F_{l'}(\delta))}$, it suffices to estimate how many pairs   $\brackets{i,i'}$ satisfy   
		\[
		\absolutevalues{\frac{i}{l} -\frac{i'}{l'}} \leq 2(n+1)^2\delta.
		\]
		By \cite[Lemma 1]{DS1941}, we see that  the number of choices of $(i,i')$ is at most $4ll'(n+1)^2\delta$. Thus the total measure is no more than $8ll'(n+1)^2\delta^2$.
	\end{proof}
	
	We write  $M_r(X)$ to denote the maximal cardinality of an $r$-separated subset of a bounded set $X$.  The following lemma shows that we may interchange $M_r$ and $N_r$  in the definition of the Assouad spectrum.
	
	\begin{lemma}[Doubling property]\label{doubling_prop}
		Fix a non-empty bounded set $X \subset \RR^2$.  Then
		\[
		N_r(C(x,R) \cap X) \approx N_{2r}(C(x,R) \cap X) \approx M_{r}(C(x,R) \cap X)
		\]
		for all  $0 < r < R$ and  $x \in \RR^2$.
	\end{lemma}

	We now fix $0 < \theta < 1$ and  a sufficiently large integer $n$. The following lemma shows that the shortest horizontal gap on every line $S_{l}$ in the strips of $S_{\delta_n(\theta), \theta}(k)$ where $\ceil{\delta_n(\theta)^{-1/3}} \leq k \leq \floor{\delta_n(\theta)^{-1/2}}$ is larger than $\delta_n(\theta)$. 
	
	\begin{lemma}[Horizontal gap estimate of collapsed strips]
		Fix  $0 < \theta < \frac{2}{3}$,  sufficiently large $n$ and $\ceil{\delta_n(\theta)^{-1/3}} \leq k \leq \floor{\delta_n(\theta)^{-1/2}}$. For all $L'_{\delta_n(\theta)}(k+1) + 1 \leq  l \leq L_{ \delta_n (\theta), \theta}(k)$, we have
		\[
		\frac{1}{l(n+1)^2} \geq \frac{\delta_n(\theta)^{\frac{2}{3}}}{n+1} \geq \delta_n(\theta).
		\]
	\end{lemma}
	
	\begin{proof}
		This  follows immediately from $n+1 \leq \delta_n(\theta)^{-1/3} = (n(n+1))^{\frac{1}{3\theta}}.$
	\end{proof}

	\begin{lemma}[Covering in collapsed strips]\label{A_spe_covering_lower_bound}
		Fix sufficiently large $n$,   $\theta \in (0,\frac{2}{3})$ and $\ceil{\delta_n(\theta)^{-1/3}} \leq k \leq \floor{\delta_n(\theta)^{-1/2}}$.  Then 
		\[
		N_{\delta_n(\theta)}(S_{\delta_n(\theta), \theta}(k)) \gtrsim \frac{\mu(F_{\delta_n(\theta), \theta}(k))}{\delta_n(\theta)}
		\]
		where $S_{\delta_{n}, \theta}(k)$ and $F_{\delta_{n}, \theta}(k)$ are as in (\ref{A_spe_strip_2}).
	\end{lemma}
	
	\begin{proof}
		By the doubling property, it suffices to prove
		\[
		N_{\frac{1}{4} \delta_n(\theta)}(S_{\delta_n(\theta), \theta}(k)) \gtrsim \frac{\mu(F_{\delta_n(\theta),\theta}(k))}{\delta_n(\theta)} .
		\]
		
		It follows from the definition of $F_{\delta_n(\theta), \theta}(k)$ that $F_{\delta_n(\theta), \theta}(k)$ is a finite union of disjoint intervals, namely, 
		$F_{\delta_n(\theta), \theta}(k) = \bigcup_{j = 1}^J I_j.$  
		For any interval $I_j$, the horizontal gap of consecutive points in 
		$\proj_x \brackets{S_{\delta_n(\theta), \theta}(k)} \cap I_j$ 
		is no more than $2\delta_n(\theta)$.
		Thus 
		$N_{\delta_n(\theta)}(I_j) \geq \frac{\mu(I_j)}{4\delta_n(\theta)},$
		and
		\[
		N_{\delta_n(\theta)}(\proj_x\brackets{S_{\delta_n(\theta), \theta}(k))} \geq \frac{\mu(F_{\delta_n(\theta), \theta}(k))}{4\delta_n(\theta)}.
		\]
		Thus  it suffices to prove
		\[
		N_{\frac{1}{4} \delta_n(\theta)}(S_{\delta_n(\theta), \theta}(k)) \gtrsim N_{\delta_n(\theta)}(\proj_x \brackets{S_{\delta_n(\theta), \theta}(k))}.
		\]
		To prove this, again by the doubling property, it suffices to prove
		\[
		N_{\frac{1}{4} \delta_n(\theta)}(S_{\delta_n(\theta), \theta}(k)) 
		\geq M_{\delta_n(\theta)}(\proj_x \brackets{S_{\delta_n(\theta), \theta}(k)})
		\geq N_{\delta_n(\theta)}(\proj_x \brackets{S_{\delta_n(\theta), \theta}(k)})
		\]
		where the notation is as in (\ref{A_spe_strip_2}). 
		
		It follows from the doubling property that
		\[
		M_{\delta_n(\theta)}(\proj_x \brackets{S_{\delta_n(\theta), \theta}(k)})
		\geq N_{\delta_n(\theta)}(\proj_x \brackets{S_{\delta_n(\theta), \theta}(k)}),
		\]
		thus for all  integers $M \leq N_{\delta_n(\theta)}(\proj_x \brackets{S_{\delta_n(\theta), \theta}(k)})$, there exist $M$ points in $\proj_x \brackets{S_{\delta_n(\theta), \theta}(k)}$, denoted by $\bigbrackets{x_1, \dots, x_M}$, such that the horizontal distance of each points is larger than $\delta_n(\theta)$. For each $x_i = \frac{p_i}{q_i}$ where $\gcd{(p_i,q_i)} = 1$, we require a closed cube of  side length $\frac{1}{4} \delta_n(\theta)$  to cover $\brackets{\frac{p_i}{q_i}, \frac{1}{q_i}}$, and  cubes used in the resulting cover of cover $\bigbrackets{\brackets{\frac{p_i}{q_i}, \frac{1}{q_i}}}_{i=1}^M$ are disjoint.  Therefore $
		N_{\frac{1}{4}\delta_n(\theta)}(S_{\delta_n(\theta), \theta}(k)) \geq M,
$
		which implies that
		\[
		N_{\frac{1}{4} \delta_n(\theta)}(S_{\delta_n(\theta), \theta}(k)) 
		\geq M_{\delta_n(\theta)}(\proj_x \brackets{S_{\delta_n(\theta), \theta}(k)})
		\]
		and the result holds.
	\end{proof}
	
	\subsection{Proof of Theorem \ref{A_spe_result}}
	
	To prove Theorem \ref{A_spe_result}, it suffices to prove the following lemma.
	
	\begin{lemma}\label{A_spe_result_lem}
		For all $0 < \theta < \frac{2}{3}$, $\dim_A^\theta  F  = \dim_A^\theta G_{f} = \frac{\frac{4}{3} - \theta}{1 - \theta}.$
	\end{lemma}
	If Lemma \ref{A_spe_result_lem} holds, then by the continuity of the Assouad spectrum, we have $	\dim_A^{\frac{2}{3}}  F  = \dim_A^{\frac{2}{3}} G_{f} = 2$, and then by \cite[Theorem 3.3.1]{F2020} we have that for all $\frac{2}{3} \leq \theta < 1$, $\dim_A^\theta  F  = \dim_A^\theta G_{f} = 2$.
	
	\begin{proof}[Proof of Upper Bound in Lemma \ref{A_spe_result_lem}]
		Fix $0 < \theta < \frac{2}{3}$.  It suffices to bound the Assouad spectrum of the full popcorn set $F$ from above. Let $R \in (0,1)$ and consider $R^{1/\theta}$ covers of an arbitrary $R$-square $C  = [x,x+R] \times [y,y+R]$ intersecting $F$.  Similar to the upper bound for the box dimension, we decompose $F \cap C $ into two regions, which we deal with  separately.  It follows by a simple cardinality estimate  that
		\begin{align*}
		& N_{R^{1/\theta}} \left( F  \cap \brackets{[x,x+R] \times [y+R^{2/(3 \theta)},y+R]}\right) \\ 
& \leq \# F  \cap \brackets{[x,x+R] \times [y+R^{2/(3 \theta)},y+R]} \\
&\lesssim R \cdot \left(R^{-2/(3 \theta)}\right)^2 = R^{1-4/(3\theta)}.
			\end{align*}
	For the remaining part of $F\cap C $
	\begin{align*}
		N_{R^{1/\theta}} \left( F  \cap \brackets{[x,x+R] \times [y,y+R^{2/(3 \theta)}]}\right) & \leq N_{R^{1/\theta}} \left(   [x,x+R] \times [y,y+R^{2/(3 \theta)}]\right)  \\
&\lesssim \left( \frac{R}{R^{1/\theta}}\right) \cdot \left(\frac{R^{2/(3 \theta)}}{R^{1/\theta}} \right) = R^{1-4/(3\theta)}
			\end{align*}
	proving $\dim_A^\theta  F \leq  \frac{\frac{4}{3} - \theta}{1 - \theta}$	as required.
	\end{proof}

	\begin{proof}[Proof of Lower Bound in Lemma \ref{A_spe_result_lem}]
		
		Fix $0 < \theta < \frac{2}{3}$. For  sufficiently large $n$, write  $R_n = \frac{1}{n(n+1)}$, $x_n = \brackets{\frac{1}{n+1}, 0}$,  $\delta_n(\theta) = R_n^{1/\theta}$ (see (\ref{delta_n_theta})) and
		\[
		C(x_n,R_n) = \medbrackets{\frac{1}{n+1}, \frac{1}{n}} \times \medbrackets{0, \frac{1}{n(n+1)}}.
		\]
		Hence, by the doubling property, it suffices to prove that  for all sufficiently small $\eps > 0$, and all sufficiently large $n$, we have
		\begin{equation}\label{A_spe_lower_bound_1}
		N_{R_n^{1/\theta}}(C(x_n, R_n) \cap G_f) \gtrsim  R_n^{-\frac{4}{3\theta} + 1 + \frac{6\eps}{\theta}}.
		\end{equation}
		Fix $\eps > 0$.  For sufficiently large $n$, it follows from the definition of $C(x_n, R_n) \cap G_f$ that
		\[
		\bigcup_{k=\ceil{\delta_n(\theta)^{-1/3}}}^{\floor{\delta_n(\theta)^{-1/2 + \eps}}} S_{\delta_n(\theta),\theta}(k) \subset C(x_n, R_n) \cap G_f
		\]
		where $S_{\delta_n(\theta), \theta}(k)$ is as in (\ref{A_spe_strip_2}).  We now use the local \DS estimate (Proposition \ref{local_D-S_esti}) and the \CE inequality (Theorem \ref{C-E_ine}) to estimate the covering number of each $S_{\delta_n(\theta), \theta}(k)$.  Fix $\ceil{\delta_n(\theta)^{-1/3}} \leq k \leq \floor{\delta_n(\theta)^{-1/2 + \eps}}$.  There exists an integer $l$ such that  
\[
		k\delta_n(\theta) \leq \frac{1}{l(n+1)} < (k+1)\delta_n(\theta)
\]
and therefore
\[
		L'_{\delta_n(\theta)}(k+1) + 1 \leq  l \leq L'_{\delta_n(\theta)}(k)
		\]
		where $L'_{\delta_n(\theta)}$ is as in  (\ref{A_spe_strip_2}).  It follows from Lemma \ref{A_spe_covering_lower_bound} , and the \CE Inequality (Theorem \ref{C-E_ine}) that
		\begin{equation}\label{A_spe_lower_bound_2}
		\begin{aligned}
		N_{R_n^{1/\theta}} (S_{\delta_n(\theta),\theta}(k))&  \gtrsim \frac{1}{\delta_n(\theta)} \cdot \mu(F_{\delta_n(\theta), \theta}(k)) \\
		& \gtrsim \frac{1}{\delta_n(\theta)}\frac{\brackets{\sum\limits_{l = L'_{\delta_n(\theta)}(k+1) + 1}^{L'_{\delta_n(\theta)}(k)} \mu(F_l(\delta_n(\theta)))}^2}{\mathop{\sum\sum}\limits_{l,l' = L'_{\delta_n(\theta)}(k+1) + 1}^{L'_{\delta_n(\theta)}(k)} \mu(F_l(\delta_n(\theta)) \cap F_{l'}(\delta_n(\theta)))}
		\end{aligned}
		\end{equation}
		where $F_l(\delta_n(\theta))$ is as in  (\ref{A_spe_strip_1}), and $S_{\delta_n(\theta),\theta}(k)$ is as in  (\ref{A_spe_strip_2}). 	We first estimate the numerator in the final expression  in \eqref{A_spe_lower_bound_2}. For all $l$,
		$ \mu(F_l(\delta_n(\theta))) \geq \delta_n(\theta) \cdot \phi(l) $,
		where $\mu$ is the Lebesgue measure on $[0,1]$ and $\phi(l)$ is the Euler totient function, see (\ref{totient_funcion}).  It follows from Theorem \ref{totient_funcion} and $\ceil{\delta_n(\theta)^{-1/3}} \leq k \leq \floor{\delta_n(\theta)^{-1/2 + \eps}}$ that for sufficiently large $n$, and all  $L'_{\delta_n(\theta)}(k+1) + 1 \leq  l \leq L'_{\delta_n(\theta)}(k)$, we have
		\[
		\log \log l \leq \log \log \frac{1}{k(n+1)\delta_n(\theta)} 
		\leq \brackets{\frac{1}{k(n+1)\delta_n(\theta)}}^\eps 
		\lesssim \delta_n(\theta)^{-2\eps}.
		\]
		Therefore  it follows from $\mu(F_l(\delta_n(\theta))) \gtrsim \delta_n(\theta) \cdot \phi(l) \gtrsim \delta_n(\theta) \cdot l \cdot (\log \log l)^{-1}$ for all large $l$ that
 \begin{align*}
		\sum\limits_{l = L'_{\delta_n(\theta)}(k+1) + 1}^{L'_{\delta_n(\theta)}(k)} \mu(F_l(\delta_n(\theta)))
		& \gtrsim \delta_n(\theta) \sum\limits_{l = L'_{\delta_n(\theta)}(k+1) + 1}^{L'_{\delta_n(\theta)}(k)} l \cdot \brackets{\log \log l}^{-1} \\
		& \gtrsim \delta_n(\theta) \sum\limits_{l = L'_{\delta_n(\theta)}(k+1) + 1}^{L'_{\delta_n(\theta)}(k)} l \cdot \brackets{\log \log \frac{1}{k(n+1)\delta_n(\theta)}}^{-1} \\
		&  \gtrsim \delta_n(\theta)^{ 1+2 \eps}\sum\limits_{l = L'_{\delta_n(\theta)}(k+1) + 1}^{L'_{\delta_n(\theta)}(k)}  l.
		\end{align*}
		We next estimate the denominator in \eqref{A_spe_lower_bound_2}.   By splitting the sum and then applying the local \DS estimate (Proposition \ref{local_D-S_esti}) to  the first sum and the trivial estimate to the second, we have
		\begin{align*}
		& \mathop{\sum\sum}\limits_{l, l' = L'_{\delta_n(\theta)}(k+1) + 1}^{L'_{\delta_n(\theta)}(k)}  \mu(F_l(\delta_n(\theta)) \cap F_{l'}(\delta_n(\theta))) \\
		& = \sum\limits_{l, l' = L'_{\delta_n(\theta)}(k+1) + 1, l \neq l'}^{L'_{\delta_n(\theta)}(k)}  \mu(F_l(\delta_n(\theta)) \cap F_{l'}(\delta_n(\theta)))
		 \ + \sum\limits_{l = L'_{\delta_n(\theta)}(k+1) + 1}^{L'_{\delta_n(\theta)}(k)}  \mu(F_l(\delta_n(\theta))) \\
		&  \lesssim n(n+1) \delta_n(\theta)^2 \cdot \sum\limits_{l, l' = L'_{\delta_n(\theta)}(k+1) + 1, l \neq l'}^{L'_{\delta_n(\theta)}(k)}  ll'   \ + \ \delta_n(\theta) \sum\limits_{l = L'_{\delta_n(\theta)}(k+1) + 1}^{L'_{\delta_n(\theta)}(k)}  l \\
		& \lesssim n(n+1) \delta_n(\theta)^2 \brackets{\sum\limits_{l = L'_{\delta_n(\theta)}(k+1) + 1}^{L'_{\delta_n(\theta)}(k)}  l}^2 \  + \  \delta_n(\theta) \brackets{\sum\limits_{l = L'_{\delta_n(\theta)}(k+1) + 1}^{L'_{\delta_n(\theta)}(k)}  l}.
		\end{align*}
		It follows from (\ref{square_esti}) that
		\[
		\sum\limits_{l = L'_{\delta_n(\theta)}(k+1) + 1}^{L'_{\delta_n(\theta)}(k)} l \approx L'_{\delta_n(\theta)}(k)^2 - \brackets{L'_{\delta_n(\theta)}(k+1) + 1}^2 \approx \frac{1}{k^3n(n+1)\delta_n(\theta)^2}.
		\]
		Thus (\ref{A_spe_lower_bound_2}) yields
		\begin{align*}
 N_{\delta_n(\theta)} (S_{\delta_n(\theta),\theta}(k)) & \gtrsim \frac{1}{\delta_n(\theta)} \cdot \frac{\delta_n(\theta)^{2+ 4\eps} \brackets{\sum\limits_{l = L'_{\delta_n(\theta)}(k+1) + 1}^{L'_{\delta_n(\theta)}(k)} l}^2}{n^2\delta_n(\theta)^2 \brackets{\sum\limits_{l = L'_{\delta_n(\theta)}(k+1) + 1}^{L'_{\delta_n(\theta)}(k)} l}^2 + \delta_n(\theta) \brackets{\sum\limits_{l = L'_{\delta_n(\theta)}(k+1) + 1}^{L'_{\delta_n(\theta)}(k)} l}}\\
		& \gtrsim \delta_n(\theta)^{-1+ 4\eps} \cdot \frac{1}{n^2} \cdot \frac{1}{k^3\delta_n(\theta)+1} \gtrsim \delta_n(\theta)^{4\eps} \cdot \frac{1}{n(n+1)} \cdot \frac{1}{k^3\delta_n^2(\theta)}.
		\end{align*}
		Then summing over $k$ from $\ceil{\delta_n(\theta)^{-1/3}}$ to $\floor{\delta_n(\theta)^{-1/2 + \eps}}$, it follows from the previous  estimate, Lemma \ref{doubling_prop} and   (\ref{square_esti}) that
		\begin{align*}
		N_{\delta_n(\theta)}(C(x_n, R_n) \cap G_f) & \geq N_{\delta_n(\theta)}\brackets{\bigcup_{k=\ceil{\delta_n(\theta)^{-1/3}}}^{\floor{\delta_n(\theta)^{-1/2 + \eps}}} S_{\delta_n(\theta),\theta}(k)}\\
		& \gtrsim \sum\limits_{k=\ceil{\delta_n(\theta)^{-1/3}}}^{\floor{\delta_n(\theta)^{-1/2 + \eps}}}  N_{\delta_n(\theta)} \brackets{S_{\delta_n(\theta), \theta}(k)}  \\
		& \gtrsim \delta_n(\theta)^{-\frac{4}{3} + 4\eps} \cdot \frac{1}{n(n+1)}  \gtrsim R_n^{-\frac{4}{3\theta} + \frac{4\eps}{\theta} + 1}.
		\end{align*}
Therefore 
\[
		\dim_A^\theta G_f \geq \frac{\frac{4}{3} - \theta - 4 \eps}{1- \theta}
		\]
		and the result holds by taking $\eps \to 0$.
	\end{proof}

\newpage

\section*{Acknowledgements}

We  thank Andrew Mitchell for helping to predict the box dimension result before we had a proof.  He used a box-counting computer programme to estimate  the covering number for very small $\delta$ and a log-log plot suggested $4/3$ should be the dimension.  We also thank an anonymous referee for making several helpful comments which simplified and improved the paper.
	
H. Chen is thankful for the excellent atmosphere for research provided by the University of St Andrews.   H. Chen was funded by China Scholarship Council (File No. 201906150102) and NSFC (No. 11601161, 11771153 and 11871227).  J. M. Fraser was  supported by an  EPSRC Standard Grant (EP/R015104/1) and a Leverhulme Trust Research Project Grant (RPG-2019-034).   H. Yu was supported by the European Research Council (ERC) under the European Union’s Horizon 2020 research and
	innovation programme (grant agreement No. 803711), and indirectly by Corpus Christi College, Cambridge.

	\bibliographystyle{amsplain}

\end{document}